\documentclass[11pt, a4paper]{article}

\usepackage[T1]{fontenc}
\usepackage[utf8]{inputenc}
\usepackage{amsmath,amssymb,amsthm,color,latexsym,enumerate,a4,graphicx}
\usepackage{theoremref}
\usepackage{xspace}
\usepackage{bbding}
\usepackage{tikz}
\usetikzlibrary{patterns,positioning,decorations}
\usepackage{comment}
\usepackage{enumitem}
\usepackage{lmodern}
\usepackage{microtype}
\usepackage{a4wide}

   \newcommand{\thlab}[1]{\thlabel{#1}} 

\theoremstyle{plain}
\newtheorem{theorem}{Theorem}
\newtheorem{lemma}[theorem]{Lemma}
\newtheorem{claim}[theorem]{Claim}

\newtheorem{question}[theorem]{Question}
\numberwithin{theorem}{section}

\theoremstyle{definition}

\newcommand{\tw}{\ensuremath{\textrm{tw}}}

\newcommand{\eps}{\ensuremath{\varepsilon}}
\renewcommand{\epsilon}{\ensuremath{\varepsilon}}

\renewcommand\tilde[1]{\widetilde{#1}}
\renewcommand\hat[1]{\widehat{#1}}
\newcommand\se{\subseteq}

\renewcommand\to{\rightarrow}
\newcommand\nto{\nrightarrow}
\newcommand\sprod[2]{\ensuremath{#1\boxtimes #2}}
\newcommand\red{\ensuremath{\mathrm{red}}}
\newcommand\blue{\ensuremath{\mathrm{blue}}}

\newcommand{\cG}{{\mathcal G}}

\renewcommand{\ge}{\geqslant}
\renewcommand{\le}{\leqslant}
\renewcommand{\geq}{\geqslant}
\renewcommand{\leq}{\leqslant}

\title{\bf The size Ramsey number of graphs\\ with bounded treewidth}
\author{
Nina Kam\v{c}ev\thanks{School of Mathematics, Monash University, Melbourne, Australia. Email: {\tt nina.kamcev@monash.edu.}}
\and
Anita Liebenau\thanks{School of Mathematics and Statistics, UNSW Sydney, NSW 2052, Australia. Email: {\tt a.liebenau@unsw.edu.au.} Supported by the Australian research council (DE170100789 and DP180103684).}
\and
David R. Wood\thanks{School of Mathematics, Monash University, Melbourne, Australia. Email: {\tt david.wood@monash.edu}. Research supported by the Australian Research Council.}
\and
Liana Yepremyan\thanks{Mathematical Institute, University of Oxford, Oxford, UK. Email: {\tt yepremyan@maths.ox.ac.uk.} Supported by ERC Consolidator Grant 647678 and by a Robert Bartnik Fellowship of the School of Mathematics, Monash University. The author would like to thank for the hospitality the School of Mathematics, Monash University, and the School of Mathematics and Statistics, UNSW, where this work was partially carried out. 
}
}

\begin{document}
\maketitle

\begin{abstract}
 A graph $G$ is \emph{Ramsey} for a graph $H$ if every 2-colouring of the  edges of $G$ contains a monochromatic copy of $H$. We consider the following question: if $H$ has bounded treewidth, is there a `sparse' graph $G$ that is Ramsey for $H$? Two notions of sparsity are considered. Firstly, we show that if the maximum degree and treewidth of $H$ are bounded, then there is a graph $G$ with $O(|V(H)|)$ edges that is Ramsey for $H$. This was previously only known for the smaller class of graphs $H$ with bounded bandwidth. On the other hand, we prove that the treewidth of a graph $G$ that is Ramsey for $H$ cannot be bounded in terms of the treewidth of $H$ alone. In fact, the latter statement is true even if  the treewidth is replaced by the degeneracy and $H$ is a tree.
\end{abstract}

\section{Introduction}

A graph $G$ is \emph{Ramsey} for a graph $H$, denoted by $G\to H$, if every $2$-colouring of the edges of $G$ contains a monochromatic copy of $H$. In this paper we are interested in how sparse $G$ can be in terms of $H$ if $G\to H$. The two measures of sparsity that we consider are the number of edges in $G$ and the treewidth of $G$.

The \emph{size Ramsey number} of a graph $H$, denoted by $\hat r (H)$, is the minimum number of edges in a graph $G$ that is Ramsey for $H$. The notion was introduced by Erd\H{o}s, Faudree, Rousseau  and Schelp~\cite{efrs1978}. Beck~\cite{beck1983} proved $\hat r(P_n)\le 900n$, answering a question of  Erd\H{o}s~\cite{e1981}. The constant 900 was subsequently improved by Bollob\'as~\cite{bollobas2001}, and by Dudek and Pra\l at~\cite{dp2015}. In these proofs the host graph $G$ is random.  
Alon and Chung~\cite{ac1988}  provided an explicit construction of a graph with $O(n)$ edges that is Ramsey for $P_n$. 

Beck~\cite{beck1983} also conjectured that  the size Ramsey number of bounded-degree trees is linear in the number of vertices, and noticed that there are trees (for instance, double stars) for which it is quadratic. 
Friedman and Pippenger \cite{fp1987} proved Beck's conjecture. The implicit constant was subsequently improved by Ke~\cite{k1993} and by Haxell and Kohayakawa~\cite{hk1995}. Finally, Dellamonica Jr \cite{d2012} proved that the size Ramsey number of a tree $T$ is determined by a simple structural parameter $\beta(T)$ up to a constant factor, thus establishing another conjecture of Beck~\cite{beck1990}. 

In the same paper, Beck asked whether all bounded degree graphs have a linear size Ramsey number, but this was disproved by R{\"o}dl and Szemer\'edi~\cite{rs2000}. They constructed a family of graphs of maximum degree 3 with superlinear size Ramsey number. 

In 1995, Haxell, Kohayakawa and \L uczak showed that cycles have linear size Ramsey number~\cite{hkl1995}. Conlon~\cite{conlon2016} 
asked whether, more generally, the $k$-th power of the path $P_n$ has size Ramsey number at most $cn$, where the constant $c$ only depends on $k$. Here the $k$-th power of a graph $G$ is obtained by adding an edge between every pair of vertices at  distance at most $k$ in $G$. Conlon's question was recently answered in the affirmative by Clemens, Jenssen, Kohayakawa, Morrison, Mota, Reding and Roberts~\cite{clemens2019}.

Their result is equivalent to saying that graphs of bounded bandwidth have linear size Ramsey number. We show that the same conclusion holds in the following more general setting. The treewidth of a graph $G$, denoted by $\tw(G)$, can be defined to be the minimum integer $w$ such that $G$ is a subgraph of a chordal graph with no $(w+2)$-clique. While this definition is not particularly illuminating, the intuition is that the treewidth of $G$ measures how `tree-like' $G$ is. For example, trees have treewidth 1. 
Treewidth is of fundamental importance in the graph minor theory  of Robertson and Seymour and in algorithmic graph theory; see \cite{bodlaender98,HW17,reed03} for surveys on treewidth.

\begin{theorem}
\thlab{thm:tw}
For all integers $k, d$ there exists $c=c(k,d)$ such that if $H$ is a graph 
of maximum degree $d$ and treewidth at most $k$, then  
$$ \hat r (H) \le c |V(H)|.$$
\end{theorem}

\thref{thm:tw} implies the above $O(|V(H)|)$ bounds on the size Ramsey number from~\cite{clemens2019}, since powers of paths have bounded treewidth and bounded degree. Powers of complete binary trees are examples of graphs covered  by our theorem, but not covered by any previous results in the
literature. Note that the assumption of bounded degree in \thref{thm:tw}  cannot be dropped in general since, as mentioned above, there are trees of superlinear size Ramsey number~\cite{beck1990}. Furthermore, the lower bound from~\cite{rs2000} implies that an additional assumption on the structure of $H$, such as bounded treewidth, is also necessary.  
We prove \thref{thm:tw} in Section~\ref{sec:FirstTheorem}.

We actually prove an off-diagonal strengthening of \thref{thm:tw}. For graphs $H_1$ and $H_2$, the \emph{size Ramsey number} $\hat{r}(H_1, H_2)$ is the minimum number of edges in a graph $G$ such that every red/blue-colouring of the edges of $G$ contains a red copy of $H_1$ or a blue copy of $H_2$. We prove that if $H_1$ and $H_2$ both have $n$ vertices, bounded degree and bounded treewidth, then $\hat{r}(H_1, H_2)\leq cn$. 
Moreover, we show that there is a host graph that works simultaneously for all such pairs $H_1$ and $H_2$ and that has bounded degree.

\begin{theorem}
\thlab{thm:offdiagonal}
For all integers $k, d\geq 1$ there exists $c=c(k,d)$ such that for every integer $n\geq 1$ there is a graph $G$ with $cn$ vertices and maximum degree $c$, such that for all graphs $H_1$ and $H_2$ with $n$ vertices, maximum degree $d$ and treewidth $k$, every red/blue-colouring of the edges of $G$ contains a red copy of $H_1$ or a blue copy of $H_2$.
\end{theorem}

The second contribution of this paper fits into the framework of \emph{parameter Ramsey numbers}: for any monotone graph parameter $\rho$, one may ask whether $\min \{\rho(G) : G\to H\}$ can be bounded in terms of $\rho(H)$. This line of research was conceived in the 1970s by Burr, Erd\H{o}s and Lov\'asz~\cite{bel76}. The usual Ramsey number and the size Ramsey number (where $\rho(G) = |V(G)|$ and $\rho(G) = |E(G)|$ respectively) are classical topics. Furthermore, the problem has been studied when $\rho$ is the clique number~\cite{folkman70,  nr76}, chromatic number~\cite{bel76,zhu98}, maximum degree~\cite{hmr14,jmw13} and minimum degree~\cite{bel76,fglps2014,fl2007,szz2010} (the latter requires the additional assumption that the host graph $G$ is minimal with respect to subgraph inclusion, otherwise the problem is trivial).

It is therefore interesting to ask whether   $\min \{\tw(G) : G\to H\}$ can be bounded in terms of $\tw(H)$. Our next theorem shows that the answer is negative, even when replacing treewidth by the weaker notion of degeneracy.  For an integer $d$, a graph $G$ is \emph{$d$-degenerate} if every subgraph of $G$ has minimum degree at most $d$. The \emph{degeneracy} of $G$ is the minimum integer $d$ such that $G$ is $d$-degenerate. 
It is well known and easily proved that every graph with treewidth $w$ is $w$-degenerate, but treewidth cannot be bounded in terms of degeneracy (for example, the 1-subdivision of $K_n$ is 2-degenerate, but has treewidth $n-1$). 

\begin{theorem}\thlab{thm:d-degenerate}
For every $d\ge 1$ there is a tree $T$ such that if $G$ is $d$-degenerate then $G\nto T$.
\end{theorem}

A positive restatement of \thref{thm:d-degenerate} is that the edges of every $d$-degenerate graph can be 2-coloured with no monochromatic copy of a specific tree $T$ (depending on $d$). This is a significant strengthening of a theorem by Ding, Oporowski, Sanders and Vertigan~\cite[Theorem~3.9]{ding1998}, who  proved that the edges of every graph with treewidth at most $k$ can be $k$-coloured with no monochromatic copy of a certain tree $T$.  We also note that a statement  similar to Theorem~\ref{thm:d-degenerate} does not hold in the online Ramsey setting, see Section~4 in~\cite{cfs2014} for more details. 

Furthermore, \thref{thm:d-degenerate} is tight in the following sense. If $\cG$ is a monotone graph class with unbounded degeneracy, then for every tree $T$, there is a graph $G\in \cG$ such that $G\to T$. Indeed, for a given tree $T$, let $G$ be a graph in $\cG$ with average degree at least $4 |V(T)|$, which exists since $\cG$ is monotone with unbounded degeneracy. In any 2-colouring of $E(G)$, one colour class has average degree at least  $2|V(T)|$. Thus there is a monochromatic subgraph of $G$ with minimum degree at least $|T|$, which contains $T$ as a subgraph by a folklore greedy algorithm.

\section{Tools}

Our proof of \thref{thm:offdiagonal} 
relies on the following characterisation of graphs with bounded treewidth and bounded degree. The {\em strong product} of graphs  $G$ and $H$, denoted by $\sprod{G}{H}$, is the graph with vertex set  $V(G)\times V(H)$, where  $(v_1,u_1)$ is adjacent to $(v_2,u_2)$ in $\sprod{G}{H}$ if $v_1=v_2$ and $u_1u_2\in E(H)$, or $v_1v_2\in E(G)$ and $u_1=u_2$, or $v_1v_2\in E(G)$ and $u_1u_2\in E(H)$. 
    Note that $\sprod{T}{K_k}$ is obtained from $T$ by replacing each vertex by a clique and replacing each edge by a complete bipartite graph.

\begin{lemma} [\cite{do95,Wood09}]
\label{StrongProduct}
    Every graph with treewidth $w$ and maximum degree $d$ is a subgraph of $T \boxtimes K_{18w d}$ for some tree $T$ of maximum degree at most $18w d^2$.
\end{lemma}

Our host graph $G$ in the proof of \thref{thm:offdiagonal} is obtained from a random $D$-regular graph $H$ on $O(n)$ vertices, for a suitable constant $D$, by taking the third power of $H$, and then replacing every vertex by a clique of bounded size, and by replacing every edge by a complete bipartite graph. To show that $G$ has the desired Ramsey properties we will exploit certain expansion properties of $H$.

An $(N, D, \lambda)$\emph{-graph} is a $D$-regular $N$-vertex graph in which every eigenvalue except the largest one is at most $\lambda$ in absolute value. The existence of graphs with $\lambda = O(\sqrt{D})$ is shown, for instance, by considering a random $D$-regular graph on $N$ vertices, denoted by $G(N, D)$.
    
    \begin{lemma} [\cite{f2008}]
    \thlab{thm:friedman}
    Let $D \geq  3$ be an integer and let $ND$ be even. With probability tending to 1 as $N\to\infty$,  
    every eigenvalue of $G(N, D)$ except the largest one is at most $2 \sqrt{D}$ in absolute value. 
    \end{lemma}

    For a graph $G$ and sets $U, W \subseteq V(G)$, let $e(U,W)$ be the number of edges with one endpoint in $U$ and the other one in $W$. Each edge with both endpoints in $U \cap W$ is counted twice. 
    We will use the following well-known estimate on the edge distribution of a graph in terms of its eigenvalues, see, e.g., \cite{ks2006} for a proof.

    \begin{lemma}[\cite{ks2006}]  \thlab{thm:ks:edgedist}
    For every $(N, D, \lambda)$-graph $G$ and for all sets $S, T\subseteq V(G)$,
            $$  \left |   e(S, T) -\frac{D|S||T|}{N}   \right | \leq\
            \lambda \sqrt{|S||T|\left(1 - \frac{|S|}{N} \right)\left(1 - \frac{|T|}{N} \right)}. $$
        \end{lemma}
    
    The key tool that we use is the following implicit result of Friedman and Pippenger~\cite{fp1987}, which shows that every $(N, D, \lambda)$-graph with the appropriate parameters is `robustly universal' for bounded-degree trees.   Let $\mathcal{T}_{n,d}$ be the set of all trees with $n$ vertices and maximum degree at most $d$. The next lemma follows implicitly from the proofs of Theorems 2 and 3 in~\cite{fp1987}.
    
    \begin{lemma}[\cite{fp1987}] \thlab{thm:fp:vertex}  
Let $\eps>0$ and $d,n$ be  integers. Let  $D$ and $N$ be integers such that  $D > 100d^2 / \varepsilon^4$   and $N >10 d^2 n / \varepsilon^2$, and let $G$ be an $(N, D, \lambda)$-graph with $\lambda = 2\sqrt{D}$.  Then every induced subgraph of $G$ on at least~$\eps N$ vertices contains every tree in $\mathcal{T}_{n,d}$.
    \end{lemma}
    
We summarise the above results in the following lemma. 

     \begin{lemma} \label{lem:originalgraph}
        For every integer $d$, every $\varepsilon>0$ and for all even $D > 100d^2 / \varepsilon^4$ there exists $c$ such that for all integers $n, N$ with $N \geq c n$ there exists an $N$-vertex $D$-regular graph $H$ with the following properties: 
        \begin{itemize}
           \item[(1)] For every pair of disjoint sets $S, T\subseteq V(H)$ with $ |S|, |T|\geq 2N / \sqrt{D} $ we have  $e(S,T)>0$.
           \item[(2)] Every induced subgraph of $H$ on at least $\varepsilon N$ vertices contains every tree in $\mathcal{T}_{n,d}$.
       \end{itemize}
    \end{lemma}
    
\begin{proof}
    Let  $D > 100d^2 / \varepsilon^4$ be an even integer and $N >10 d^2 n / \varepsilon^2$. Let $H$ be an $(N, D, \lambda)$-graph where $\lambda = 2\sqrt{D}$,  which exists by~\thref{thm:friedman}. Property (2) follows from~\thref{thm:fp:vertex}.    Moreover, for all sets $S , T\subseteq V(H)$ with $ |S|, |T|\geq 2N / \sqrt{D} $ we have $\lambda \sqrt{|S||T|} <  \frac{D |S| |T| }{N}$, which implies  $e(S, T)>0$ by  \thref{thm:ks:edgedist}, as desired.
\end{proof}

We also need the following lemma of Friedman and Pippenger~\cite{fp1987}. For a graph $H$ and $X\subseteq V(H)$, let $\Gamma_H(X)$ be the set of vertices in $V(H)$  
adjacent to some vertex in $X$. 

\begin{lemma}[\cite{fp1987}]\thlabel{lem:friedman} 
If $H$ is a non-empty graph such that  for each $X\subseteq V(H)$ with $1\leq |X|\leq 2n-2$, 
$$|\Gamma_H(X)|\geq (d+1)|X|$$
then $H$ contains every tree in $\mathcal{T}_{n,d}$.
\end{lemma}

Finally, we need the following standard tools.

\begin{lemma}[K\"{o}vari, S\'{o}s, Tur\'{a}n~\cite{KST1954}]
\label{KST}
Every graph  with $n$ vertices and no $K_{s,s}$ subgraph has at most $(s-1)^{1/s} n^{2-1/s} + (s-1)$ edges. 
\end{lemma}

\begin{lemma}[Lov\'{a}sz Local Lemma~\cite{EL75}]
\label{LLL}
Let $\mathcal{E}$ be a set of events in a probability space, each with probability at most $p$, and each mutually independent of all except at most $d$ other events in $\mathcal{E}$. If $4pd\leq 1$ then with positive probability no event in $\mathcal{E}$ occurs. 
\end{lemma}

\section{Proof of \thref{thm:offdiagonal}}
\label{sec:FirstTheorem}

We start with the following lemma that states that if a graph does not contain all trees in~$\mathcal{T}_{n,d}$ then its complement contains a complete multipartite subgraph where the parts have `large' size.  
In fact, our proof shows that if the second assertion does not hold, (i.e.~there is no complete multipartite graph with large parts in the complement), then the graph contains a `large' expander as a subgraph. The containment of every tree in $\mathcal{T}_{n, d}$ then follows from \thref{lem:friedman}. 
Statements of similar flavour are also proved and utilised in~\cite{POKROVSKIY2018, POKROVSKIY2017384}. 

\begin{lemma}\label{aux:treeorpartite}
Fix integers $n, d, q$ and let  $N \geq 20ndq$. In every red/blue-colouring of $E(K_N)$ there is either a blue copy of every tree in $\mathcal{T}_{n,d}$, or a red copy of a complete $q$-partite graph in which every part has size at least $\frac{N}{5dq}$. 
\end{lemma}

\begin{proof}
Let $G$ be the spanning subgraph of $K_N$ consisting of all the blue edges. We may assume that $G$ does not contain every tree in $\mathcal{T}_{n, d}$. 
By \thref{lem:friedman}, for every non-empty set $S\subseteq V(G)$, there exists $X\subseteq S$ such that $1\leq |X|\leq 2n-2$ and $|\Gamma_{G[S]}(X)|<(d+1)|X|$. 
Note that for such $S$ and $X$, all the edges of $K_N$ between $X$ and $S\setminus(X\cup \Gamma_{G[S]}(X))$ must be red. 
Let $S_1, S_2, \dots, S_{m+1}$ and $X_1,X_2,\dots,X_{m}$ be sets of vertices in $G$ such that $S_1=V(G)$ and $S_{m+1}=\emptyset$, and 
for $1\leq i\leq m$:
\begin{itemize}
\item $X_i\subseteq S_i$ with $1\leq |X_i|\leq 2n-2$ and $|\Gamma_{G[S_i]}(X_i)|<(d+1)|X_i|$, and 
\item $S_{i+1} = S_i \setminus ( X_i \cup \Gamma_{G[S_i]}(X_i))$. \end{itemize}
We stop when $S_{m+1} = \emptyset$. Note that $X_1,X_2,\dots,X_{m}$ are pairwise disjoint. 
Since all the edges of $K_N$ between $X_i$ and $S_{i+1}$ are red, all the edges between distinct $X_i$ and $X_j$ are red. Let $X=\bigcup_{i=1}^m{X_i}$. Note that 
$$N=\sum_{i=1}^m|X_i \cup \Gamma_{G[S_i]}(X_i)|< \sum_{i=1}^m(d+2)|X_i| = (d+2)|X|.$$
Thus $|X| > \frac{N}{d+2}$. 

We now combine the parts $X_i$ to reach the required size. Let $Y_1 = X_1 \cup X_2 \cup \dots \cup X_j$, where $j$ is the minimal index such that $|X_1 \cup X_2 \cup \dots \cup X_j| \geq \frac{N}{5dq}$. 
Since $|X_i|\leq 2n-2<\frac{N}{10dq}$, we have the upper bound, $|Y_1| < \frac{3N}{10dq}$. Repeating  the same argument starting at $X_{j+1}$, and noting that $|X| > \frac{N}{d+2}\geq  q \cdot \frac{3N}{10dq}$, we construct  $Y_1, Y_2, \dots Y_q$, satisfying $ |Y_i| \geq \frac{N}{5dq}$ and such that all edges between any distinct $Y_i$ and  $Y_j$ are red.

\end{proof}

Let $T$ be a rooted tree with root $r$. For each vertex $v$ of $T$, let $p_T(v)$ denote the \emph{parent} of $v$, where for convenience we let $p_T(r)=r$. Let $p^2_T(v)$ denote the \emph{grandparent} of $v$; that is, $p^2_T(v)=p_T(p_T(v))$. We denote the set of \emph{children} of $v$ by $C_T(v)$, and define $C_T^2(v)=C_{T}(v) \cup \left( \bigcup_{x\in C_T(v)}{C_{T}(x)}\right)$ to be the set of children and \emph{grandchildren} of $v$. Let $d_T(v)$ be the distance between $r$ and $v$, that is,  the number of edges on the path from $r$ to $v$. For each integer $i$, let $L_i(T)$ be the set of vertices $v$ with $d_T(v)=i$.   In the above definitions, we may omit the subscript $T$ if $T$ is clear from the context.

Given a tree $T$ rooted at $r$, define another tree $T'$  rooted at $r$ as follows. The vertex set of $T'$ is defined to be $\{r\}\cup\bigcup_{i\geq 0} L_{2i+1}(T)$. A pair $vw$ with $v, w \in V(T')$ is an edge of $T'$ if $p^2_T(v)=w$ or $p^2_T(w)=v$. In particular, $C_{T'}(r)=C_T(r)$.. We call $T'$  the \emph{truncation} of $T$. Note that if $T$ has maximum degree $d$, then $T'$ has maximum degree at most $d^2$.

Let $s$ and $m$ be  integers. Suppose we are given a graph $G$, a vertex partition $(V_1, V_2, \dots, V_m)$ of $G$, and an edge-colouring $\psi:E(G)\rightarrow \{\red,\blue\}$. Define an auxiliary colouring of the complete graph $K_m$ with vertex set $[m]$ as follows. For distinct $i,j\in[m]$,  colour the edge $ij$ blue if there is a blue $K_{s,s}$  between $V_i$ and $V_j$ in $G$, and red otherwise. We call this edge-colouring the \emph{$(G,\psi,s)$-colouring of $K_m$.} This auxiliary coloring also appears in~\cite{allenetal}, and subsequently in~\cite{clemens2019}.

\begin{lemma}~\thlab{lem:chopping}
Fix integers $n$, $d$, $k$, $m$. Let $T$ be a tree in $\mathcal{T}_{n,d}$ rooted at $x_0$, and let $T'$ be the truncation of $T$. Let $s=(d+d^2)k$. Suppose we are given a graph $G$, a vertex partition $(V_1, V_2, \dots, V_m)$ of $G$, and an edge-colouring $\psi:E(G)\rightarrow \{\red,\blue\}$ such that, for all $i\in[m]$, all the edges of $G[V_i]$ are present and are blue, and $|V_i|\geq s$. If there exists a blue copy of $T'$ in the $(G,\psi,s)$-colouring of $K_m$, then there exists a blue copy of $T\boxtimes K_k$ in $G$.
\end{lemma}

\begin{figure}[h] \label{fig:truncation}
    \centering
    \includegraphics{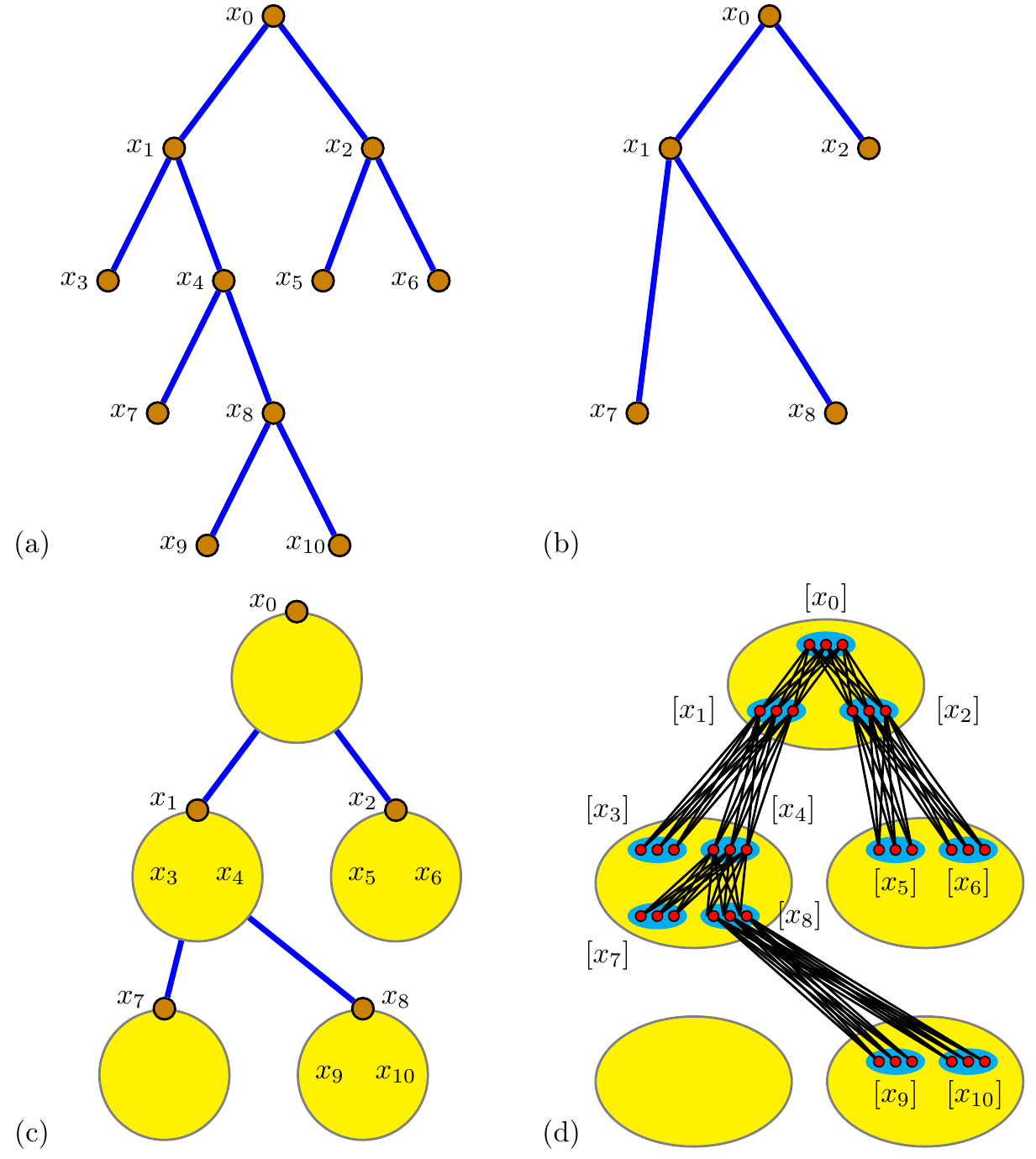}
    \caption{(a) tree $T$, (b) truncation $T'$, (c) the corresponding bags, (d) embedding of $T\boxtimes K_k$ where $[x_i]$ means $\{x_i\}\times K_k$}
    \label{Truncation}
\end{figure}

\begin{proof}
Let $\varphi$ be the $(G,\psi,s)$-colouring of $K_m$ and suppose  $g: V(T')\rightarrow [m]$ is an embedding of $T'$ in the blue subgraph of $K_m$. 
Let $x_0,x_1,x_2,\dots ,x_{m'}$ be the vertices of $V(T')$ ordered by their distance from the root $x_0$ in $T'$. We will find a blue copy of $T \boxtimes K_k$ whose vertices are in $V_{g(x_i)}$ for $i= 0, \dots, m'$. We warn the reader that in this proof we often use notation $f(S\boxtimes K_k)$ to denote the image of $S\boxtimes K_k$,  for some subset $S\se V(T)$,   under some  embedding $f$ into $G$, without precisely defining how $f$ acts on each vertex of $S\boxtimes K_k$, but rather claiming that such an embedding exists. This is done for brevity, and to keep the proof intuitive.

We define a collection $\{B_x : x\in V(T')\}$ of subsets of $V(T)$ as follows.
Let $B_{x_0}=\{x_0\}$ and for each $x\in V(T')\setminus\{x_0\}$, let $B_{x}=\{x\}\cup C_T(x)$. We call $B_x$  the \emph{bag} of the vertex $x$. Observe that the bags are pairwise disjoint, and they partition the entire vertex set $V(T)$. They will help us keep track of the  embedding of $T\boxtimes K_k$ in $G$.

We will proceed iteratively, starting from the root $x_0$ and following the order of the vertices $x_i$ we fixed earlier.  At each step $i$, we will have a partial embedding $f_i$ of $T_i\boxtimes K_k$ in $G$, where $T_i$ is the subtree $T[\cup_{j\leq i}{B_{x_j}}]$. Our final embedding will be $f=f_{m'}$.  At step $0$ we will embed $B_{x_0}\boxtimes K_k$ in some way; this will define $f_0$.  At step $i\geq 1$, $f_i$ will be defined as an extension of $f_{i-1}$,  and the extension will be defined only on $B_{x_i}\boxtimes K_k$ so that the image of the latter  `links' back appropriately to the embedding of $T_{i-1} \boxtimes K_k$. Before specifying our embedding, we list the properties that $f$ will satisfy:

\begin{itemize}
\item[(P1)] $f(T\boxtimes K_k)\subseteq \bigcup_{x\in V(T')}{V_{g(x)}}$,
\item [(P2)] $f((\{x_0\}\cup C_T(x_0))\boxtimes K_k) \se V_{g(x_0)}$, 
\item [(P3)] for every $x\in V(T')\setminus\{x_0\}$, $f(C_T^2(x) \boxtimes K_k) \se V_{g(x)}$, 
\item[(P4)] every edge of $f(T\boxtimes K_k)$ will be coloured blue.
\end{itemize}
 
 Note that (P2) implies that at most $(d+1)k$ vertices are embedded in $V_{g(x_0)}$, and every other $V_{g(x)}$ (with $x\neq x_0$) will contain at most $(d+d^2)k$ embedded vertices by (P3).  
 Moreover, (P4) will be  satisfied for edges of   $f(T\boxtimes K_k)$ embedded inside one partition class $V_j$. To guarantee  that those edges of $f(T\boxtimes K_k)$ that go between distinct partition classes $V_j$ and $V_k$ are blue, we will make use of the properties of the auxiliary colouring $\varphi$.  Finally, we define our iterative embedding  scheme from which properties (P1)--(P4) can be easily read out, thus completing the proof.

\textbf{Step 0:}  Let $T_0=\{x_0\}$ and embed $T_0\boxtimes K_k$ into $V_{g(x_0)}$, by picking any $k$ vertices in $V_{g(x_0)}$; this determines $f_0$. Recall that all edges inside $V_{g(x_0)}$ are blue, hence indeed this is a valid embedding of  $T_0\boxtimes K_k$.

\textbf{Step $\boldsymbol{ i\geq 1}$:} Having defined $f_{i-1}$, we now show how to extend it to $f_i$. Recall that $B_{x_i}=\{x_i\}\cup C_T(x_i)$. Let $y$ be the grandparent of $x_i$.  Since there is an edge $x_iy$ in $T'$ and since $g$ is a blue embedding of $T'$ in $K_m$, there is a blue $K_{s,s}$ between $V_{g(x_i)}$ and $V_{g(y)}$. Let $L$ be any such copy of $K_{s,s}$.  Define $f_i$ on $\{x_i\}\boxtimes K_k$ to be a set of any $k$ vertices in $V_{g(y)}\cap V(L)$  disjoint from the image of $f_{i-1}$. Define $f_i$ on $C_T(x_i)\boxtimes K_k$ to be any set of $k |C_T(x_i)|$ vertices in  $V_{g(x_i)}\cap V(L)$  disjoint from the image of $f_{i-1}$. This is possible since $|V_{g(x_i)} \cap V(L)| \geq s= (d^2+d)k$, and the total number of vertices embedded into $V_{g(x_i)}$ during the procedure is at most $(d^2+d)k$.
\end{proof}

The next lemma is a standard application of the Lov\'asz Local Lemma. Given a graph $F$ let $F(t)$ denote the \emph{blowup} of $F$ where each vertex $v$ is replaced by an independent set $I(v)$ of size $t$, and each edge $uv$ is replaced by a complete bipartite graph between $I(u)$ and $I(v)$.

\begin{lemma}~\thlab{lem:lifting}
Fix $t\geq 1$. Let $F$ be a graph with maximum degree $\Delta$. Let $F'$ be a spanning subgraph of $F(t)$  such that for every edge $vw\in E(F)$ there are at least $(1-\frac{1}{8\Delta})t^2$ edges in $F'$ between $I(v)$ and $I(w)$. Then $F\subseteq F'$.
\end{lemma}

\begin{proof}
For each vertex $v$ of $F$, independently choose a random vertex $v'$ in $I(v)$. For each edge $vw$ of $F$, let $E_{vw}$ be the event that $v'w'$ is not an edge of $F'$. Since there are at least $(1-\frac{1}{8\Delta})t^2$ edges between $I(v)$ and $I(w)$, the probability of $E_{vw}$ is at most $\frac{1}{8\Delta}$. Each event $E_{vw}$ is mutually independent of all other events, except for the at most $2\Delta$ events corresponding to edges incident to $v$ or $w$.  Since $4(\frac{1}{8\Delta})(2\Delta) \leq 1$, by Lemma~\ref{LLL}, 
the probability that some event $E_{v,w}$ occurs is strictly less than 1. 
Thus, there exist choices for $v'$ for all $v\in V(F)$, such that $v'w'$ is an edge of $F'$ for every edge $vw$ of $F$. This yields a subgraph of $F'$ isomorphic to $F$. 
\end{proof}

Both \thref{thm:tw} and \thref{thm:offdiagonal} are implied by Lemma~\ref{StrongProduct} and the following result.

\begin{theorem} \thlab{thm:mainsizezramsey}
For all integers $k, d$ there exists $c=c(k,d)$ such that for all $n$ there is a graph $G$ with $cn$ vertices and maximum degree $c$, such that for all trees $T_1$ and $T_2$ with $n$ vertices and maximum degree $d$, every red/blue-colouring of $E(G)$ contains a red copy of $\sprod{T_1}{K_k}$ or a blue copy of $\sprod{T_2}{K_k}$.
\end{theorem}

\begin{proof}
Let $\eps= (d^{2}(2k+1)2^{2k+4})^{-1}$. Let $D$ be  the smallest even number larger than $100d^2/\eps^4$. Let $c$ be derived from Lemma~\ref{lem:originalgraph}  applied with this choice of $\eps$, $d$ and $D$. Let $N=\max\{cn, 40nd^2(2k+1)\}$ and let  $H$ be any $N$-vertex $D$-regular graph derived from Lemma~\ref{lem:originalgraph}. Set $s = (d^2+d)k$ and $t=(64kd)^s$. Denote the Ramsey number of $t$ by $r(t)$. Recall that $H^3$ is a graph on the same vertex set as $H$ where $uv$ is an edge in $H^3$ whenever $u$ and $v$ are at distance at most three in $H$.  Let $G = \sprod{H^3}{K_{r(t)}}$.
 
Since $H$ is $D$-regular, $H^3$ has maximum degree at most $D^3$, and $G$ has maximum degree at most $D^3 r(t)+ r(t)-1$. Let $A(v)$ denote the copy of $K_{r(t)}$ corresponding to $v\in V(H)$. Let $\psi:E(G)\rightarrow \{\red, \blue\}$ be any edge-colouring of $G$. We will show that it must contain either a red copy of $T_1 \boxtimes K_k$ or a blue copy of $T_2\boxtimes K_k$. 

By definition of $r(t)$, for each vertex $v\in V(H)$, $A(v)$ contains a monochromatic copy of $K_t$, say on vertex set $B(v)$. Let $W$ be the set of vertices $v\in V(H)$ for which $B(v)$ induces a blue $K_t$. By symmetry between $T_1$ and $T_2$, we may assume that $|W| \geq \frac {1}{2} |V(H)|$. Let  $N'=|W| \geq \frac {N}{2}$.
    
Let $B(W) = \bigcup_{v \in W} B(v)$ and let $\varphi$ be  the $(G[B(W)], \psi, s)$-colouring of $K_{N'}$.  Root $T_2$ at an arbitrary vertex. Let $T_2'$ be the truncation of $T_2$. If there is a blue copy of $T_2'$ in $K_{N'}$, then \thref{lem:chopping} implies that $G[B(W)]$ contains a blue copy of $T_2 \boxtimes K_k$, so we are done.
    	
We henceforth assume that there is no blue copy of $T_2'$ in $K_{N'}$. Since $T_2'$ has maximum degree at most $d^2$ and $N' \geq 20nd^2(2k+1)$ there are sets $V_0,V_1,\dots, V_{2k}\subseteq V(K_{N'})$ of size at least $\frac{N'}{5d^2(2k+1)}$ such that all the edges in $K_{N'}$  between two distinct parts $V_i$ and $V_j$ are red, by Lemma~\ref{aux:treeorpartite}.
   
For $i\in[2k]$, define an \emph{$i$-matching} to be a matching of edges in $H$ with one endpoint in $V_0$ and the other in $V_i$. (Note that we are now considering the original graph $H$ not $K_{N'}$.)\ We will find a set $S \subseteq V_0$ satisfying $|S|>2^{-{2k}}|V_0|$, and a collection of $i$-matchings  $\{M_i\}_{i=1}^{2k}$  such that each $M_i$ covers $S$. We proceed by induction on $i$.  Assume  at the end of step $j\leq 2k-1$  we have  found a set $S_j \subseteq V_0$ with $|S_j|>2^{-{j}}|V_0|$ and a collection of $i$-matchings $\{M_i\}_{i=1}^{j}$, where each $M_i$ covers $S_j$. At step $j+1$, let $M_{j+1}$ be a maximum matching   between $S_j$ and $V_{j+1}$. If $M_{j+1}$ consists of fewer than $|S_j|/2$ edges, then, by K\H{o}nig's theorem, the bipartite graph between $S_j$ and $V_{j+1}$ has a vertex cover of order at most $|S_j|/2 $. But then we can find sets $X \subset S_j$ and $Y \subset V_{j+1}$ with $e_{H}(X, Y)=0$ and  $|X|, |Y| \geq |S_j|/2 \geq 2^{-2k-2}|V_0| > \eps N$. This contradicts property (1) from Lemma~\ref{lem:originalgraph}. Hence $M_{j+1}$ covers at least $|S_j|/2 \geq |V_0|\cdot 2^{-(j+1)}$ vertices of $S_j$. We set $S_{j+1}=V(M_{j+1})\cap S_j$ and proceed. After $2k$ steps, we reach the desired set $S_{2k}$, which we call $S$.
  
For each vertex $v  \in S$, for $i\in[2k]$, let $v_i \in V_i$ be the unique neighbour of $v$ in $M_i$. Since $|S|>2^{-2k}|V_0|> \eps N$, $H[S]$ contains a copy $\tilde{T}_1$ of $T_1$ on some vertex set $U$ by property (2) from Lemma~\ref{lem:originalgraph}.   Next we show that there is a red copy of $T_1 \boxtimes K_{k}$ in $K_{N'}$ contained in the vertex set  of $\tilde{T}_1 \cup \{M_i\}_{i=1}^{2k}$ and use this copy to find a red copy  of  $T_1 \boxtimes K_{k}$ in $G[B(W)]$ via~\thref{lem:lifting}.
        
Root  $\tilde{T}_1$  at any vertex $\tilde{r}$.  For  every vertex $v\in V(\tilde{T}_1)$ let $S(v) = \{v_1,v_2,\dots,v_k \}$ if $v$ is at even distance from $\tilde{r}$ and $S(v) = \{v_{k+1},v_{k+2},\dots,v_{2k} \} $, otherwise. Note that for any $u,v\in V(\tilde{T}_1)$, the sets $S(u)$ and $S(v)$ are disjoint. Moreover, for every $v\in V(\tilde{T}_1)$, $S(v)$ induces a red clique in $K_{N'}$ because the vertices of $S(v)$  are elements of distinct partition classes~$V_i$. If $u$ and $v$ are adjacent in $\tilde{T}_1$, then also edges between $S(u)$ and $S(v)$  are red in $K_{N'}$ since  all the vertices of $S(u) \cup S(v)$ lie in distinct partition classes $V_i$.  So this shows that the vertex set $\bigcup_{v \in U}{S(v)}$ induces a red copy  $\tilde{T_1 \boxtimes K_{k}}$ of   $T_1 \boxtimes K_{k}$ in $K_{N'}$. It remains to `lift' this copy to the graph $G[B(W)]$ with the colouring $\psi$.  First we observe that every edge in  $\tilde{T_1 \boxtimes K_{k}}$  is in fact an edge of $H^3$.  Indeed, for any $u\in V(\tilde{T}_1)$, and any $i\neq j$,  $u_i,u_j\in S(u)$  are at distance at most two in $H$, hence $u_iu_j$ is an edge in $H^3$. Now if $u$ and $v$ are adjacent in $\tilde{T}_1$, then for any $u_i\in S(u)$ and $v_j\in S(v)$, the distance between $u_i$ and $v_j$ in $H$ is at most $3$, so $u_i$ and $v_j$ are also adjacent in $H^{3}$.

Recall that if $uv$ is an edge of $H^{3}$ and $\varphi(uv)$ is red in $K_{N'}$, then the complete bipartite graph $G_{uv}$ between $B(u)$ and $B(v)$ in $G$ contains no blue copy of $K_{s, s}$. Lemma~\ref{KST} implies that $G_{uv}$ has at most $ (s-1)^{1/s} t^{2-1/s} + (s-1) \leq 4t^{2-1/s}$ blue edges. Note that $4t^{2-1/s} \leq \frac{t^2}{16dk}$. Let $F= \tilde{T_1 \boxtimes K_{k}}$ and let $F'$ be the subgraph of $G$ consisting of all the red edges of $G_{uv}$ over all $uv\in E(F)$. It is now easy to see that $F$ and $F'$ satisfy the assumptions of~\thref{lem:lifting}. Therefore $G$ contains a red copy of $T_1 \boxtimes K_k$ which finishes the proof.
\end{proof}

\section{Proof of \thref{thm:d-degenerate}}
\label{sec:SecondTheorem}

Let $T_{d, h}$ be the complete $d$-ary tree of height $h$ with a root vertex $r$; that is, every non-leaf vertex has exactly $d$ children and every leaf is at distance $h$ from $r$. \thref{thm:d-degenerate} is implied by the following.  Recall that, for a rooted tree $T$, $d_T (v)$ denotes the number of edges of the path from the root to $v$  in $T$.

\begin{theorem} \thlab{thm:main}
For every integer $i\geq 1$, every $(2^i-1)$-degenerate graph $G$ is not Ramsey for the tree $T_{2^{i+1},2^{i}}$. 
\end{theorem}

\begin{proof}
We proceed by induction on $i$. For $i = 1$, $G$ is a tree, so fix an arbitrary vertex to be the root of $G$ and colour the edges of $G$ by their distance to the root modulo 2 (where the distance of an edge $uv$ to the root $r$ is $\min\{d_G(u),d_G(v)\}$). There is no monochromatic path of length 3, and in particular no monochromatic copy of $T_{4,2}$.

Now let $i\ge 2$ and set $d= 2^i$ and $h=2^{i-1}$ for brevity. 
Let $G$ be a $(d-1)$-degenerate graph. It follows from the definition of degeneracy that $G$ has a vertex-ordering $v_1, v_2, \dots, v_n$, such that each vertex $v_j$ has at most $d-1$ neighbours $v_k$ with $k < j$. 
Form an oriented graph $\vec G$ by choosing the orientation $(v_j,v_k)$ for an edge $v_jv_k\in E(G)$ if $j<k$. 
Then each vertex has in-degree at most $d-1$. 

We now partition $V(G)$ into sets $V_r$ and $V_b$ such that both $G[V_r]$ and $G[V_b]$ are $(d/2-1)$-degenerate. 
Start by assigning $v_1$ to $V_r$. 
For $j=2,3,\dots,n$, assume that $v_1, v_2, \dots, v_{j-1}$ have been assigned to $V_r$ or $V_b$. 
Add $v_j$ to $V_r$ if $V_r$ contains at most $d/2-1$ of the in-neighbours of $v_j$. Otherwise add it to $V_b$. Note that in the latter case, $V_b$ contains at most $d/2-1$ of the in-neighbours of $v_j$, since $v_j$ has at most $d-1$ in-neighbours in $\vec G$. Clearly, this does not affect the in-degree of $v_1, v_2, \dots, v_{j-1}$ in $\vec G[V_r]$ and $\vec G[V_b]$. 
Thus, this process produces the desired sets $V_r$ and $V_b$.

 By induction, there is a red/blue-colouring $\psi'$ of the edges in $E_G(V_r) \cup E_G(V_b)$ not containing a monochromatic copy of $T_{d,h}$. 
We extend $\psi'$ to a red/blue-colouring  $\psi$ of $E(G)$ in the following way. 
For an edge $uv\in E_G(V_r,V_b)$ assume without loss of generality that it is directed from $u$ to $v$ in $\vec G$, that is $(u,v)\in \vec G$. Then colour $uv$ red if $u\in V_r$, and blue if $u \in V_b$. In other words, the edge $uv$ `inherits' the colour from its source vertex in $\vec G$. 

We claim that there is no monochromatic copy of $T_{2d, 2h}$ in this colouring of $E(G)$. Assume the opposite and let ${\tilde T_{2d, 2h}}$ be a monochromatic copy of $T_{2d, 2h}$ in G. For each vertex $v$ in $T_{2d, 2h}$, we denote its copy in ${\tilde T_{2d, 2h}}$ by $\tilde v$. Without loss of generality we may assume that ${\tilde T_{2d, 2h}}$ is red. 

\begin{claim} \label{claim:aux1}
If $\tilde v$ is a non-leaf vertex of $\tilde T_{2d, 2h}$ that lies in $V_b$, then there are at least $d$ children 
$\tilde u_1, \dots, \tilde u_{d}$ of $\tilde v$ in $\tilde T_{2d, 2h}$ such that $\tilde u_j\in V_b$ for all $j\in [d]$. 
\end{claim} 
\begin{proof} The number of children of the vertex $\tilde v$ in  $\tilde T_{2d, 2h}$ is $2d$.  Out of these, the number of children $w$ such that $(w,\tilde{v})\in \vec G$ is at most $d-1$. Furthermore, each edge $(\tilde v, w)\in \vec G$ with $ w \in V_r$ is coloured blue in $\psi$, by definition and since $\tilde v \in V_b$. That implies that none of these edges can be part of ${\tilde T_{2d, 2h}}$. It follows that at least $d+1$ neighbours of $\tilde v$ in $\tilde T_{2d, 2h}$ are elements of $V_b$. At most  one of these vertices is the parent of $\tilde v$, and the claim follows. 
\end{proof}
Recall that $\tilde r$ is the root of $\tilde T_{2d, 2h}$. 
\begin{claim}
For every vertex $\tilde v \in V(\tilde T_{2d, 2h})$ at distance at most $h$ from $\tilde r$ in $\tilde T_{2d, 2h}$ we have that $\tilde v\in V_r$.  
\end{claim}

\begin{proof} 
Assume that $\tilde v \in V_b$ and has distance at most $h$ in $\tilde T_{2d, 2h}$ from $\tilde r$. Apply Claim~\ref{claim:aux1} iteratively to $\tilde v$ and all of its descendants $\tilde u$ that lie in $V_b$. In $h$ iterations (before reaching the leaves of $\tilde T_{2d, 2h}$), we construct a red copy of $T_{d, h}$ whose vertices all lie in $V_b$; that is,  a red copy of $T_{d, h}$. This contradicts the property of $\psi'$.
\end{proof}

It follows that all vertices in $\tilde T_{2d, 2h}$ at distance at most $h$ from $\tilde r$ must lie in $V_r$, forming a red copy of $T_{d,h}$ in $G[V_r]$, which again contradicts the property of $\psi'$. \end{proof}

After the first preprint of this paper was finished we learned~\cite{rollinPersonal} that Maximilian Gei\ss er, Jonathan Rollin and Peter Stumpf independently obtained a proof of 
\thref{thm:d-degenerate}. This proof is unpublished, yet short and nice, so we include their argument here. 

\begin{proof}[Second proof of \thref{thm:d-degenerate}]
Let $G$ be a $d$-degenerate graph. We show that  $G$ is not Ramsey for $T_{d+1,d+1}$. 
Assume without loss of generality that the vertex set of $G$ is $[n]$, for some $n$, and that every $u\in V(G)$ has at most $d$ neighbours $v$ with $v<u$. 
Let $\phi:V(G)\to [d+1]$ denote a proper colouring of the vertices of $G$ using at most $d+1$ colours.  Define an edge colouring $\psi$ by colouring an edge $uv$ with $u < v$ red if $\phi(u) < \phi(v)$ and blue otherwise. A path $v_1 \ldots v_n$ in $G$ is called monotone if its vertices are ordered $v_1 < \ldots < v_n$. Each monochromatic monotone path in $\psi$ has at most $d$ vertices, since the colours of its vertices are either increasing or decreasing under $\phi$. On the other hand each copy of $T_{d+1,d+1}$ in $G$ contains a monotone path on $d$ vertices, since each inner vertex $u$ has a child $v$ with $u<v$. Hence there are no monochromatic copies of $T_{d+1,d+1}$ in $G$.
\end{proof}

\section{Concluding Remarks}

We have shown that for a graph $H$ of bounded maximum degree and treewidth,   there is a graph $G$ with $O(|V(H)|)$ edges that is Ramsey for $H$. It is now natural to ask whether the size Ramsey number of a planar graph $H$ of bounded degree is linear in $|V(H)|$. A first candidate to consider is the grid graph. 
Recently Clemens, Miralaei, Reding, Schacht and Taraz~\cite{grid} have shown that  the size Ramsey  number of the grid graph on $n\times n$ vertices is bounded from above by $n^{3+o(1)}$. There are no non-trivial lower bounds known. 

\begin{question}  Is the size Ramsey number of the grid graph on $n\times n$ vertices $O(n^2)$?
\end{question}

Furthermore, we propose a multicolour extension of our result.
\begin{question}
Given positive integers $w,d,n, s\geq 3$ and an $n$-vertex graph $H$ of maximum degree $d$ and treewidth $w$, does there exist $C=C(w,d,s)$ and  a graph $G$ with $Cn$ edges such that every $s$-colouring of the edges of $G$ contains a monochromatic copy of $H$? 
\end{question} When $H$ is a bounded-degree tree, a positive answer (and even a stronger density analogue result) follows from the work of Friedman and Pippinger~\cite{fp1987}. Han, Jenssen, Kohayakawa, Mota and Roberts~\cite{hjkmr2018} have recently shown that the above extension holds for graphs of bounded bandwidth (or, equivalently, for any fixed power of a path).

Our second result is that the edges of every $d$-degenerate graph can be 2-coloured without a monochromatic copy of a fixed tree $T=T(d)$. The maximum degree of $T$ in the proof of \thref{thm:main} is $2d+1$. It follows from \cite[Lemma 5]{mp2013} that $T$ cannot be replaced by a tree whose maximum degree is bounded by an absolute constant which is independent of $d$.

Ding, Oporowski, Sanders and Vertigan~\cite{ding1998} also showed that for every tree $T$, there is a graph $G$ of treewidth two such that every red/blue-colouring of the edges of $G$ contains a red copy of $T$ or a blue copy of a subdivision of $T$. We wonder whether the following generalisation is true. 
\begin{question}
	Is there a function $f(k)$ with the following property: for every graph $H$ of treewidth $k$, there is a graph $G$ of treewidth $f(k)$ such that every red/blue-colouring of the edges of $G$ contains a red copy of $H$ or a blue copy of a subdivision of $H$?
\end{question}

\bigskip
\noindent
{\bf Acknowledgement:} We would like to thank Jonathan Rollin for sending us the alternative proof of \thref{thm:d-degenerate} and for pointing us to \cite{mp2013}. Following the original release of our paper, the paper~\cite{Motaetal} was posted to the arXiv. It provides an affirmative answer to Question 5.2.

  \let\oldthebibliography=\thebibliography
  \let\endoldthebibliography=\endthebibliography
  \renewenvironment{thebibliography}[1]{%
    \begin{oldthebibliography}{#1}%
      \setlength{\parskip}{0.1ex}%
      \setlength{\itemsep}{0.1ex}%
  }{\end{oldthebibliography}}

\bibliographystyle{abbrv}
\bibliography{references}

\end{document}